\documentclass[11pt]{amsart}
\usepackage{amsfonts}
\usepackage{amssymb}
\usepackage{amsmath}
\usepackage{dsfont}
\usepackage{color}
\usepackage{graphicx}
\newtheorem{theorem}{Theorem}[section]

\newtheorem{proposition}[theorem]{Proposition}

\theoremstyle{definition}
\newtheorem{df}{Definition}
\newtheorem{example}[df]{Example}
\newtheorem{remark}[df]{Remark}
\newtheorem{problem}[df]{Problem}

\newcommand{\N}{\mathbb N}

\newcommand{\R}{\mathbb R}
\newcommand{\Z}{\mathbb Z}

\newcommand{\ve}{\varepsilon}
\newcommand{\on}{\operatorname}
\newcommand{\Var}{\on{Var}}

\newcommand{\mc}{\mathcal}

\title{Cutting sets of continuous functions on the unit interval}
\author{Marek Balcerzak}
\address{Institute of Mathematics,
         Lodz University of Technology, ul. W\'olcza\'nska 215,
         93-005 \L\'od\'z,
         Poland}
         \email{marek.balcerzak@p.lodz.pl}
         
         \author{Piotr Nowakowski}
\address{Institute of Mathematics,
         Lodz University of Technology, ul. W\'olcza\'nska 215,
         93-005 \L\'od\'z,
         Poland}
         \address{Institute of Mathematics, Czech Academy of Sciences,
\v{Z}itn\'a 25, 115 67 Prague 1, Czech Republic}   
         \email{piotr.nowakowski@edu.p.lodz.pl}
     
\author{Micha{\l} Pop{\l}awski}
\address{Department of Exact and Natural Sciences
		Jan Kochanowski University in Kielce, ul. Uniwersytecka 7, 
    	25-406 Kielce,
        Poland}

\email{michal.poplawski.m@gmail.com}
\date{}
\thanks{Piotr Nowakowski was supported by the GA \v{C}R project 20-22230L and RVO: 67985840.}
\begin{document}
\begin{abstract}
For a function $f\colon [0,1]\to\R$, we consider the set $E(f)$ of points at which $f$ cuts the real axis. 
Given $f\colon [0,1]\to\R$
and a Cantor set $D\subset [0,1]$ with $\{0,1\}\subset D$, we obtain conditions equivalent to the conjunction $f\in C[0,1]$ (or $f\in C^\infty [0,1]$) 
and $D\subset E(f)$. This generalizes some ideas of Zabeti. We observe that,
if $f$ is continuous, then $E(f)$ is a closed nowhere dense subset of $f^{-1}[\{ 0\}]$ where each $x\in \{0,1\}\cap E(f)$ is an accumulation point of $E(f)$. Our main result states that, for a closed nowhere dense set $F\subset [0,1]$ with each $x\in \{0,1\}\cap E(f)$ being an accumulation point of $F$, there exists $f\in C^\infty [0,1]$ such that $F=E(f)$.
\end{abstract} 
\maketitle
\section{Introduction}
Our research has been inspired by the studies by Zabeti \cite{Z}.
According to \cite{Z}, we say that a function $f\colon [0,1]\to\R$ cuts the real axis at a point $x\in [0,1]$ provided that $f(x)=0$, and in every neighbourhood $U$ of $x$ there exist $y,z\in U$ such that 
$f(y)<0<f(z)$. Zabeti in \cite{Z} showed
two examples of continuous functions cutting the real axis at every point of a measurable set 
of positive measure.
In the first example, a continuous function is of infinite variation, and in the second example, 
a function has derivatives of all orders.

Given a function $f\colon [0,1]\to\R$, the set of all points $x\in [0,1]$ at which $f$ cuts the real axis
will be called the cutting set for $f$ and denoted by $E(f)$.

\begin{proposition} \label{P0}
If $f\colon [0,1]\to\R$ is a continuous function then $E(f)$ is a closed nowhere dense
subset of $f^{-1}[\{0\}]$. Additionally, each $x\in\{0,1\}\cap E(f)$ is an accumulation point of $E(f)$.
\end{proposition}
\begin{proof}
%Since $f$ is continuous, the set $f^{-1}[\{0\}]$ is closed.
Asssume that a sequence $(x_n)$ of points in $E(f)$ is convergent to $x\in [0,1]$.
Then $f(x)=0$ by the continuity of $f$. Fix a neighbourhood $U$ of $x$ and find $x_n\in U$.
Pick a neighbourhood $V\subset U$ of $x_n$. Then there are points $y,z\in V$ such that $f(y)<0<f(z)$.
Hence $x\in E(f)$. Thus $E(f)$ is closed.

Suppose that $E(f)$ is dense in some nondegenerate interval $I$. 
Since $f$ is continuous, it should be constantly equal to zero on $I$ which contradicts the fact that it cuts the real axis 
at every point of $E(f)$. Consequently, $E(f)$ is nowhere dense. The final assertion follows from
the definition of $E(f)$.
\end{proof} 

As it is shown in \cite[Theorem 1]{Z}, there exists a continuous function $f\colon [0,1]\to\R$
with $E(f)$ containing a Cantor-type set of positive measure.
This result will be discussed and generalized in Section 2.
Observe that, by Proposition \ref{P0}, the set $E(f)$ cannot be of full measure in any nondegenerate interval.

If we do not assume that $f$ is continuous, a situation is different.

\begin{example}
Consider (see \cite{KS}) a differentiable function $g\colon [0,1]\to\R$ such that for any interval 
$(a,b)\subset [0,1]$ there exist points $x_1,x_2\in(a,b)$ such that $g'(x_1)<0$ and $g'(x_2)>0$.
This shows that the function $g$ is nowhere monotone in $[0,1]$. Observe that $E(g')=\{ x\colon g'(x)=0\}$.
Also, the set $E(g')$ is dense. Indeed, take any open interval $(a,b)\subset [0,1]$ and pick $x,y\in (a,b)$ such that $g'(x)<0<g'(y)$. Since the derivative $g'$ has the Darboux property (see \cite {Ru}), there exists a point $z$ between $x$ and $y$ such that $g'(z)=0$. It is known that the derivative $g'$ is Baire 1 (we have $g'(x)=\lim_{n\to\infty} n(g(x+1/n)- g(x))$ for
$x\in [0,1)$). Hence $\{ x\colon g'(x)=0\}$ is of type $G_\delta$ 
(see \cite[Theorem 5.14]{Go}). Thus $E(g')$ is a comeager set.
\end{example}

 The article is organized as follows. In Sections 2 and 3, we discuss conditions under which a Cantor set is contained in $E(f)$ for $f\in C[0,1]$ and $f\in C^\infty[0,1]$, respectively. This refers to Theorems 1 and 2 of Zabeti from his paper \cite{Z}. Section 4 contains the main result where we reverse the situation described in Proposition \ref{P0}. Namely,
 for a closed set $F\subset [0,1]$ with the respective properties, we construct a function $f\in C^\infty[0,1]$ such that $E(f)=F$.
 In Section 5, we formulate an open problem concerning the measure of the cutting set of a typical function in $C[0,1]$.
 The Appendix presents a proof of the result, being a byproduct of our considerations, which states that the complement of the set of functions in $C[0,1]$ with infinite variation is a $\sigma$-porous set.

\section{Cantor sets included in the cutting sets of continuous functions}
By the Brouwer theorem \cite[Theorem 7.4]{Ke}, a topological space is homeomorphic to the classical 
ternary Cantor set in $[0,1]$
if and only if, it is perfect, nonempty compact metrizable and zero-dimensional.
Among subspaces of $\R$ those spaces are exactly nonempty perfect, compact and nowhere dense sets.
Such sets will be called Cantor sets in $\R$.  

Zabeti in \cite{Z} considered a subclass of Cantor subsets of $[0,1]$. Because of a construction, they are usually called central Cantor sets. Every central Cantor set is generated by a sequence $(\xi_n)_{n\ge 0}$ with positive real terms where $\xi_0:=1$,
and $0<\xi_{n+1}<(1/2)\xi_n$ for all $n\ge 0$. In the first step, we remove from $[0,1]$ the open middle interval
of length $\xi_0-2\xi_1$. Then we obtain a set $D_1$ as the union of two closed intervals $[0,\xi_1]$, $[1-\xi_1,1]$,
So, the Lebesgue measure $\lambda(D_1)$ is $2\xi_1$. Having constructed $D_n$ as the union of $2^n$ constituent intervals of length $\xi_n$, we remove
from each of them the open middle interval of length $\xi_n-2\xi_{n+1}$. Then we obtain a set $D_{n+1}$ that is the union of $2^{n+1}$ pairwise disjoint closed intervals of length $\xi_{n+1}$. The intersection $D:=\bigcap_n D_n$ is the central Cantor set generated by $(\xi_n)_{n\ge 0}$. Clearly, the measure $\lambda(D)$ equals $\lim\limits_{n \to \infty} 2^n\xi_n$.

Given a function $f\colon [0,1]\to\R$ and a closed interval $I\subset [0,1]$, we denote by $\Var (f,I)$, the variation of $f$ restricted to $I$. In particular, $\Var(f, [0,1])$ is the total variation of $f$ on $[0,1]$.
A result of \cite{Z} shows, for a given central Cantor set $D$ of positive measure $\alpha\in (0,1)$, an example of a continuous function $f\colon [0,1]\to\R$ with infinite variation on $[0,1]$, that cuts the real axis at every point of $D$.
One of our aims is to obtain a strengthened description of this phenomenon.
Our approach is more general. Firstly, we consider all Cantor sets in $[0,1]$, not necesserily central Cantor sets of positive measure as it was done in \cite{Z}. Secondly, we will observe that the main idea of the Zabeti example characterizes all continuous functions $f$ cutting the real axis at any point of a Cantor set, and in fact, we can make $f$ of finite or infinite variation as we want. 

Note that, in the respective interpretation, continuous functions of infinite variation on 
$[0,1]$
are typical since they form a comeager set in the Banach space $C[0,1]$ with the supremum norm. This fact is known
since a function of finite variation on $[0,1]$ is differentiable almost everywhere, and continuous nowhere differentiable
functions on $[0,1]$ form a comeager set in $C[0,1]$; see \cite[Section~11]{Ox}. However, we propose a stronger version
of this fact. 

%for the reader's convenience, we will prove it as 
%the following proposition with a short direct proof.
Let us recall the notions of porous and $\sigma$-porous sets in a metric space (see \cite{Za}, \cite{Za1}).
Let $X$ be a metric space. By $B(x,r)$ we will denote the open ball with center $x \in X$ and radius $r > 0$.
Given $E\subset X$, $x \in X$ and $R >0$, we write $\gamma(x,R,E):=\sup\{r > 0 \colon \exists_{z \in X} \,\,B(z,r) \subset B(x,R)\setminus E\}$. 
We define the porosity of $E$ at $x$ as
\begin{equation} \label{por}
p(E,x):=2 \limsup\limits_{\ve \to 0^+}  \frac{\gamma(x,\ve,E)}{\ve}.
\end{equation}
We say that a set $E$ is porous if its porosity is positive at each $x \in E$. We say that $E$ is $\sigma$-porous if it is a countable union of porous sets. Observe that every porous set is nowhere dense, so every $\sigma$-porous set is meager.

\begin{theorem}\label{BV}
The subset $U$ of $C[0,1]$ that consists of continuous functions with infinite variation on $[0,1]$ is a $G_\delta$
comeager set in this space. Moreover, its complement is $\sigma$-porous.
\end{theorem}

We postpone a proof of this result to the Appendix.

Now, let us present the promised characterization extending the first example of Zabeti \cite{Z}.

\begin{theorem} \label{T1}
Fix $f\colon [0,1]\to\R$. Let $D\subset [0,1]$ be a Cantor set with $\{0,1\}\subset D$, and let $W_n$, $n\in\N$, 
be a one-to-one sequence of all connected components of $[0,1]\setminus D$. Then $f\in C[0,1]$ with $D\subset E(f)$ if and only if the following conditions hold simultaneously
\begin{itemize} 
\item[(i)] $f|_D=0$, functions $f|_{\on{cl}(W_n)}$, $n\in\N$, are continuous, and the series
$\sum_{n} \chi_{\on{cl}(W_n)} f$ is uniformly convergent on $[0,1]$;
\item[(ii)] $D\subset\on{cl}\left(\bigcup_{n\in M^-}W_n\right)\cap \on{cl}\left(\bigcup_{n\in M^+}W_n\right)$  where
$M^-:=\{ n\colon W_n\cap f^{-1}[(-\infty,0)]\neq\emptyset\}$ and 
$M^+:=\{ n\colon W_n\cap f^{-1}[(0,\infty)]\neq \emptyset\}$.
%\item[(iii)] $\sum_{n\in\N}\Var(f,\on{cl}(W_n))=\infty.$
\end{itemize}
\end{theorem}
\begin{proof}
{\bf Necessity.} To show (i) it suffices to prove that the series
$\sum_{n}\chi_{\on{cl}(W_n)}f$ is uniformly convergent on $[0,1]$. Suppose it is not the case.
So, the uniform Cauchy condition fails. Hence there exist $\ve>0$, infinitely many indices
$n_1<n_2<\dots$ and points $x_j\in W_{n_j}$ such that $|f(x_j)|\ge\ve$ for all $j\in\N$. Pick a convergent subequence 
$x_{k_j}\to x$. Since $\lambda(W_{n_j})\to 0$, we have $x\in D$. But $f(x)=0$ which together with $|f(x_{k_j})|\ge\ve$ for $j\in\N$
contradicts the continuity of $f$ at $x$.

To show (ii) note that, since $D\subset E(f)$, around every point $x\in D$ one can find points
$y$ and $z$ with $f(y)<0<f(z)$. Since $f|_D=0$, we have $y\in\bigcup_{n\in M^-}W_n$ and $z\in\bigcup_{n\in M^+}W_n$.

{\bf Sufficiency.} From (i) it follows that $f$ is continuous on $[0,1]$. Let $x\in D$. Then $f(x)=0$ and by (ii),
in every neighbourhood of $x$ we can pick points $y,z\in\bigcup_{n\in\N}W_n$ such that
$f(y)<0<f(z)$ as desired.
\end{proof}

A variant of the above theorem, where one requires the infinite variation of $f\in C[0,1]$, can be stated as follows.

\begin{theorem} \label{T1prim}
Fix $f\colon [0,1]\to\R$. Let $D\subset [0,1]$ be a Cantor set with $\{0,1\}\subset D$, and let $W_n$, $n\in\N$, 
be a one-to-one sequence of all connected components of $[0,1]\setminus D$. Then $f\in C[0,1]$ with $\Var(f,[0,1])=\infty$  and $D\subset E(f)$ if and only if conditions \em{(i), (ii)} stated in Theorem \ref{T1} hold together with
\begin{itemize} 
%\item[(i)] $f|_C=0$, functions $f\chi_{\on{cl}(B_n)}$, $n\in\N$, are continuous, and the series
%$\sum_{n} f\chi_{\on{cl}(B_n)}$ is uniformly convergent on $[0,1]$;
%\item[(ii)] $\bigcup_{n\in M}B_n$ is dense in $[0,1]$ where
%$$M:=\{ n\colon B_n\cap f^{-1}[(0,\infty)]\neq \emptyset\;\land\; B_n\cap f^{-1}[(-\infty,0)]\neq\emptyset\}.$$
\item[(iii)] $\sum_{n\in\N}\Var(f,\on{cl}(W_n))=\infty.$
\end{itemize}
\end{theorem}
\begin{proof}
Since $f|_D=0$, it is enough to observe that
$\Var(f,[0,1])=\sum_{n\in\N}\Var(f,\on{cl}(W_n))$.
The rest is a consequence of the previous theorem.
\end{proof}

\begin{remark}
The example given in \cite[Theorem 1]{Z} follows the scheme of sufficiency in Theorem \ref{T1prim}.
Namely, the author considers a central Cantor set $D$ generated by a sequence $(\xi_n)_{n\ge 0}$ with $\lambda(D)>0$.
If $W=(a,b)$ is any open interval that has been removed from $[0,1]$ in the $n$-th step of the construction of $D$, the value $f(x)$ for $x\in W$ is simply defined
as $c_n\sin\left(\frac{2\pi x}{b-a}\right)$ where $c_n>0$ with
\begin{equation} \label{E1}
\sum_n c_n<\infty\;\; \mbox{and} \lim_{n\to\infty}2^n c_n=\infty .
\end{equation}
Also, let $f(x):=0$ for each $x\in D$.
The first part of (\ref{E1}) guarantees that the series $\sum_{k} f\chi_{\on{cl}(W_k)}$ is uniformly convergent on $[0,1]$
(cf. condition (i)), and the second part of (\ref{E1}) forces that the total variation of $f$ on $[0,1]$ is infinite.
%$\Var(f, [0,1])=\infty$. 
Indeed, note that $\Var(f, \on{cl}(W))=2c_n$
and then it suffices to use (iii).
The definition of $f$ on $W$ clearly implies that condition (ii) holds.

In the case of an arbitrary Cantor set $D\subset [0,1]$ with $\{0,1\}\subset D$, one can mimic this construction. 
Firstly, we order the sequence
of all open components $W_n$, $n\in\N$, of $[0,1]\setminus D$ with respect to decreasing lenghts
(if a few intervals have the same length, they can be ordered arbitrarily).
Using this ordering, we define $f|_{\on{cl}(W_n)}$, $n\in\N$, step by step taking care to obtain (i) and (ii). We let $f(x):=0$ for $x\in D$.
Additionally,  if we want to control $V(f, [0,1])$, we should use condition (iii).
\end{remark}

\section{The case of functions of class $C^\infty [0,1]$}
We denote by $C^\infty [0,1]$ the vector space of functions $f\colon [0,1]\to\R$ that have derivatives of all 
orders on $[0,1]$. This is a complete space when we use the metric $d$ given by
$$d(f,g)=\sum_{n=0}^\infty2^{-n}\min\{1, ||f^{(n)}(x)-g^{(n)}(x)||\}$$
where $||\cdot ||$ denotes as before, the supremum norm.
In \cite[Theorem 2]{Z}, the author constructed a function $f\in C^\infty [0,1]$ and a measurable set $D\subset [0,1]$ 
(of type $G_\delta$) of positive measure (that can be arbitarily close to 1) such that $D\subset E(f)$.
In this case, the choice of a set $D$ is different from that proposed previously in \cite[Theorem 1]{Z} since $D$ is not a Cantor set.
However, we have observed that the previous technique still works after the respective modification.
Namely, the counterpart of Theorem \ref{T1} holds in the following version.

\begin{theorem} \label{T2}
Fix $f\colon [0,1]\to\R$. Let $D$ be a Cantor set in $[0,1]$ with $\{0,1\}\subset D$, and let $W_n$, $n\in\N$, be a one-to-one sequence of all connected components of $[0,1]\setminus D$. Then $f$ is in $C^\infty [0,1]$ with $D\subset E(f)$ if and only if the following conditions hold simultaneously
\begin{itemize} 
\item[(i)] $f|_D=0$, functions $f|_{\on{cl}(W_n)}$, $n\in\N$, are in $C^\infty ({\on{cl}(W_n)})$, and the series
$\sum_{n} \chi_{\on{cl}(W_n)}f^{(p)}$, for $p\ge 0$, are uniformly convergent on $[0,1]$;
\item[(ii)] $D\subset\on{cl}\left(\bigcup_{n\in M^-}W_n\right)\cap \on{cl}\left(\bigcup_{n\in M^+}W_n\right)$  where
$M^-:=\{ n\colon W_n\cap f^{-1}[(-\infty,0)]\neq\emptyset\}$ and 
$M^+:=\{ n\colon W_n\cap f^{-1}[(0,\infty)]\neq \emptyset\}$.
\end{itemize}
\end{theorem} 
The proof is analogous to that for Theorem \ref{T1}.

The following example modifies the idea of \cite[Theorem 2]{Z}. We repeat the method constructing the respective 
function $f\in C^\infty [0,1]$. Using the sufficiency part of Theorem \ref{T2}, we will ensure that, for a given Cantor set 
$D\subset [0,1]$, $f$ can be chosen so that $D\subset E(f)$. Therefore, we connect somehow the ideas of the two results stated in \cite{Z}.

\begin{example} \label{exx}
Consider the following function $h\colon \R\to\R$ which has derivatives of all orders on $\R$
and obtains non-zero (positive) values exactly on $(0,1)$:
$$h(x):=\left\{\begin{array}{ll}
e^{-x^{-2}} e^{-(x-1)^{-2}}, & \text{ if } x\in (0,1);\\
0, & \text{ otherwise}.
\end{array}\right.$$
Fix a Cantor set $D$ in $[0,1]$ with $\{0,1\}\subset D$, and let $W_n$, $n\in\N$, be a one-to-one sequence of all connected components of $[0,1]\setminus D$. With any $W_n:=(a_n,b_n)$, we associate the function $f_n\colon [0,1]\to\R$ given by
\begin{equation} \label{sinus} 
f_n(x):=c_n h\left(\frac{x-a_n}{b_n-a_n}\right)\sin\left(\frac{2\pi(x-a_n)}{b_n-a_n}\right)
\end{equation}
where a sequence $(c_k)$ will be chosen later. Then $f_n|_{\on{cl}(W_n)}\in C^\infty(\on{cl}(W_n))$ and, thanks to the factor with the function sinus, it attains negative and positive values on $W_n$. Denote $\ve_n:=b_n-a_n$.

Fix an integer $p\ge 0$. By the Leibniz rule, we can calculate the $p$-th derivative of $f_n$ at a point $x\in W_n$ as follows
$$f_n^{(p)}(x)=c_n\ve_n^{-p}(2\pi)^p\sum_{k=0}^p\binom{p}{k}h^{(k)}\left(\frac{x-a_n}{\ve_n}\right)(2\pi)^{-k}
\sin^{(p-k)}\left(\frac{2\pi(x-a_n)}{\ve_n}\right).$$
Note that 
\begin{equation} 
||f_n^{(p)}||\le c_n\ve_n^{-p}M_p\;\;\mbox{ for all }n\in\N
\end{equation}
 where $M_p:=(2\pi)^p\sum_{k=0}^p\binom{p}{k}(2\pi)^{-k}||h^{(k)}||$.
 Hence every series $\sum_{n=1}^\infty f_n^{(p)}$, $p\ge 0$, is uniformly convergent on $[0,1]$ whenever 
 $(c_n)$ is chosen so that $\sum_{n=1}^\infty c_n\ve_n^{-p}<\infty$ since then the Weierstrass test works.
 For instance, the choice $c_n:=\frac{1}{n^2}\exp\left(\frac{-1}{\ve_n}\right)$ is good since
 $\lim_{x\to 0^+}x^{-p}\exp\left(\frac{-1}{x}\right)=0$.
 
 Therefore, we have checked that conditions (i) and (ii) of Theorem \ref{T2} are satisfied if $f:=\sum_{n=1}^\infty f_n$, 
 and so, $f\in C^\infty [0,1]$ with $D\subset E(f)$.
\end{example}

\section{Main result}
An interesting question is connected with the converse to the situation described in Proposition~\ref{P0}.
Namely, given a nonempty closed nowhere dense set $F\subset [0,1]$ with $\{0,1\}\cap F$ consisting of accumulation points of $F$, we need to 
find a function $f\in C[0,1]$ 
(or even, $f\in C^\infty [0,1]$) such that $E(f)=F$.
Note that  Example \ref{exx} does not give a solution to this problem in the case if $F$ is a Cantor set with 
$\{0,1\}\subset F$ since functions $f_n$, defined in (\ref{sinus}), produce some points in $E(f)$ which are not 
in $F$.

Our main theorem will give a positive answer to the above question with a function of class $C^\infty[0,1]$.

Let $I\subset\R$ be a bounded interval. By $l(I),r(I)$ and $|I|$ we will denote the left endpoint of $I$, 
the right endpoint of $I$ and the length of $I$, respectively.
We say that an interval $I$ is a neighbour of an interval $J$ if $I$ and $J$ have a common endpoint. 
%Then $I$ is called a neigbouring interval of $J$ and vice versa. 
We denote $\{0,1\}^{<\N}:=\bigcup_{n\ge 0}\{0,1\}^n$.
\begin{theorem} \label{Cantor}
Let $F\subset [0,1]$ be a closed nowhere dense set with $\{0,1\}\cap F$ (if nonempty) consisting of accumulation points 
of $F$. Then there exists a function $f\in C^\infty [0,1]$ with
$E(f)=F$.
\end{theorem}
%%%%%%%%%%%%%%%%%%%%%%%%%%%%%%%%%%%%%%%%%%%%%%%%%

\begin{proof}
Fix $F$ as in the assumption of the theorem. 
Since $F$ is closed, (by the classical Cantor-Bendixson theorem) it can be expressed as a union of a perfect set $D$ and a countable set $Q$. Assume that $D\neq\emptyset$ since otherwise, the proof is simpler. The set $D$ is a Cantor set as a perfect and nowhere dense subset of $[0,1]$.
We will express $D$ in the fashion similar to that used for the 
Cantor ternary set. Let $I$ be a minimal closed interval containing $D$. Let $J$ denote the leftmost longest connected component of $[0,1]\setminus D$.
This is called a gap of level $0$. Denote by $I_{0}$ and $I_{1}$ (respectively) the left and the right
components of $I\setminus J$; they will be called basic intervals of level $1$.
Inside the intervals $I_{k}$, for $k=0,1$, we pick the leftmost longest components of $I_{k}\setminus D$;
we denote them $J_{k}$ and call gaps of level $1$.
The further construction goes inductively. For any $n\in\N$, we obtain the family of basic intervals $I_s$,
$s\in\{0,1\}^n$ of level $n$ and gaps $J_s$ of level $n$ inside of them.
Note that $D=\bigcap_{n\in\N}D_n$ where $D_n$ is the union of basic intervals of level $n$.
For every $x\in D$ there is a unique sequence $(x_1,x_2,\dots)\in\{0,1\}^\N$ such that
$\{ x\}=\bigcap_{n\in\N}I_{x|n}$ where $x|n:=(x_1,\dots, x_n)$ for $n\in\N$.

To simplify the further reasoning suppose that $l(I) = 0$ and $r(I) =1$. If it is not the case, then we just consider two additional intervals $(0,l(I))$ and $(r(I),1)$ that can play a role of gaps.
%The sequence $(x_n)$ will be called the address of $x$ and will be identified with $x$.

We can present the countable set $Q$ as $\{x_n\}_n \cup \{y_n\}_n$ where $x_n$'s are isolated points in $F$, and $y_n$'s are accumulation points of the set $\{x_n\}_n$. (Symbols $\{x_n\}_n$ and $\{y_n\}_n$ mean that these sets are countable; they may be finite or even empty.)
Observe that each element $x_n$ belongs to some interval $J_s$. 
%Moreover, in every $J_s$ there is countably many elements $x_n$ (maybe finitely many). 
For a fixed $s$, the set $J_s \setminus F = J_s \setminus Q$ is open, so it can be presented as the union of a countable (possibly finite) disjoint family of open components numbered as $J_{s,(i)}$ with $i$'s belonging either to the whole $\N$ or to an initial segment of $\N$.

Consider the function $h\in C^\infty [0,1]$ as in Example \ref{exx}.
%$$h(x):=\left\{\begin{array}{ll}
%e^{-\frac{1}{x^2-x}} & \text{ if } x\in (0,1);\\
%0 & \text{ otherwise},
%\end{array}\right.$$ 
Given an interval $P=[a,b]$,
define $h_P\colon [0,1]\to\R$ by $h_P(x):=h\left(\frac{x-a}{b-a}\right)$. %We know that $h\in C^{\infty}[0,1]$.

Fix $k \in \N \cup \{0\}$ and $s \in \{0,1\}^{k}$ (for $k=0$, $s$ is equal to the empty sequence, so $J_s = J$). Consider the interval $J_s$. To any interval $J_{s,(i)}$ we assign a function $f_{s,(i)}:=\pm c_{s,(i)} h_{\on{cl}J_{s,(i)}}$ where 
$c_{s,(i)} := \frac{1}{n^2 2^i}\exp\left(\frac{-1}{|J_{s,(i)}|}\right)$. The key point is the choice of the signs $\pm$. 
If $J_{s,(1)} = J_s$ (that is, if there are no $x_n$'s and $y_n$'s in $J_s$), then let $f_{s,(1)}:=-c_{s,(1)} h_{\on{cl}J_{s,(1)}}$
if the length of $s$ is odd, and $f_{s,(1)}:=c_{s,(1)} h_{\on{cl}J_{s,(1)}}$
if the length of $s$ is even.  

In the remaining cases we proceed as follows. 
The set $J_s\setminus\{ y_n\}_n$ is open, so it is the union of a countable disjoint family $\{V_m\colon m\in M\}$ of open components. Note that every interval $J_{s,(i)}$ is contained in exactly one interval $V_m$. So, it is enough to fix $V_m$ and show the choice of sign of functions $f_{s,(i)}$ for all intervals $J_{s,(i)}$ contained in $V_m$. 
Denote by $\mc V_m$ the family of
all such intervals. Observe that $V_m=(V_m\cap\{ x_n\}_n)\cup\bigcup \mc V_m$. 
Intervals $J_{s,(i)}$ in $\mc V_m$ can be ordered with respect to the following ordering:
$K\preceq L$ iff $K=L$  or there are  $K_1,\dots , K_l\in\mathcal V_m$  with  
$K=K_1$, $L=K_l$ and  $r(K_i)=l(K_{i+1})$  for $i=1,\dots, l-1$.
Then $\preceq$ on $\mc V_m$ is isomorphic to the usual ordering on exactly one of the following sets:\\
$1^0$ $\{ 1,\dots ,j\}$ for some $j\in\N$, if the family $\mathcal V_m$ is finite;\\
$2^0$ $-\N$, if $l(V_m)$ is a right-hand point of accumulation of $x_n$'s in $V_m$, and $r(V_m)$ is not a 
left-hand point of accumulation of $x_n$'s in $V_m$;\\
$3^0$ $\N$, if $l(V_m)$ is not a right-hand point of accumulation of $x_n$'s in $V_m$, and $r(V_m)$ is a
left-hand point of accumulation of $x_n$'s in $V_m$;\\
$4^0$ $\Z$, if $l(V_m)$ is a right-hand point of accumulation of $x_n$'s in $V_m$, and $r(V_m)$ is a 
left-hand point of accumulation of $x_n$'s in $V_m$.

In each of these cases, by simple induction, we alternate the signs of functions $f_{s,(i)}$ for neighbouring intervals 
$J_{s,(i)}$ in $\mc V_m$. For instance, in case $4^0$, we renumber intervals $J_{s,(i)}$ in $\mc V_m$ (according to the respective isomorphism) as $J^*_k$, $k\in\Z$, and choose the sign of the respective $f_{s,(i)}$ for $J^*_0$ as $+$. Then choose the sign $-$ for $J^*_1$, $J^*_{-1}$, next we choose the sign $+$ for $J^*_2$, $J^*_{-2}$ and so on.

By the above method, we have defined all the functions $f_{s,(i)}$ for the intervals $J_{s,(i)}$ in such a way that they have different signs on any two neighbours.

Now, fix an interval $J_{s,(i)}$ with $s\in\{0,1\}^n$, and fix an integer $p\ge 0$. We can calculate the $p$-th derivative of $f_{s,(i)}$ at a point $x\in J_{s,(i)}$ as follows
$$f_{s,(i)}^{(p)}(x)=\pm c_{s,(i)}\frac{1}{|J_{s,(i)}|^p} h^{(p)}\left(\frac{x-l(J_{s,(i)})}{|J_{s,(i)}|}\right).$$
Note that 
\begin{equation} \label{main}
||f_{s,(i)}^{(p)}||\le c_{s,(i)}\frac{||h^{(p)}||}{|J_{s,(i)}|^p} =\frac{1}{n^2 2^i}\exp\left(\frac{-1}{|J_{s,(i)}|}\right) \cdot \frac{||h^{(p)}||}{|J_{s,(i)}|^p} \le \frac{M_p}{n^2 2^i}\;\;\mbox{ for all } i, 
\end{equation}
where $M_p =||h^{(p)}|| \cdot \sup\limits_{x\in[0,1]} \frac{e^{\frac{-1}{x}}}{x^p} < \infty$. Hence every series $\sum_i f_{s,(i)}^{(p)}$, $p\ge 0$, is uniformly convergent on $[0,1]$ by the Weierstrass test. Therefore, the function $f_s:= \sum_i f_{s,(i)}$ is in $C^{\infty}[0,1]$ because every function $f_{s,(i)}$ is in $C^{\infty}[0,1]$.
Also $g_{n}:= \sum_{s\in\{0,1\}^n} f_{s}$ is in $C^{\infty}[0,1]$.
From (\ref{main}) it follows that, for any $p \ge 0$ and $n\in\N$, we have
$$
||g_{n}^{(p)}||\le \frac{M_p}{n^2}.
$$
Hence every series $\sum_{n=1}^{\infty} g_{n}^{(p)}$, $p\ge 0$, is uniformly convergent on $[0,1]$ since the Weierstrass test works. Therefore, the function $f:= \sum_{n=0}^{\infty} g_{n}$ is in $C^{\infty}[0,1]$.

We will show that $F=E(f)$.
First, fix $x\in F$ and a neighbourhood $U$ of $x$. If $x \in D$, then, by the construction, $f(x)=0$.
Also we can find two gaps $J'$ and $J''$ obtained in the construction of $D$
such that $J'\cup J''\subset U$ where $J'$ is of odd level and $J''$ is of even level.
But there are $y' \in J'$ and $y'' \in J''$ such that $f(y')<0$ and $f(y'')>0$, so $x\in E(f)$. 

If $x = x_n$ for some $n$, then also $f(x) = 0$. Moreover, $x_n = r(J_{s,(i)})=l(J_{s,(j)})$ for some $s \in \{0,1\}^{<\N}$ and $i,j \in \N$. The intervals $J_{s,(i)}$ and $J_{s,(j)}$ are neighbours, thus $f$ attains positive values on one of them and negative values on the other. Hence $x \in E(f)$.

If $x=y_m$ for some $m$, then $f(x)=0$ and there is an increasing sequence of indices $(k_n)$ such that $x_{k_n} \to y_m$. Without loss of generality we can assume that all of the terms $x_{k_n}$ belong to the interval $J_s$ for some $s \in \{0,1\}^{<\N}$. Since $f$ is continuous, then $f(y_m) = \lim\limits_{n\to\infty} f(x_{k_n}) = 0.$ Moreover, since $U$ contains infinitely many points $x_{k_n}$, there are some neighbouring intervals $J_{s,(i)}$, $J_{s,(j)}$ contained in $U$. Thus $f$ attains positive values on one of them and negative on the other. Hence $x \in E(f)$.

On the other hand, if $x \notin F$, then $x \in J_s \setminus (\bigcup_{n}\{x_n\}_n \cup \bigcup_{n}\{y_n\}_n) $ for 
some $s \in \{0,1\}^{<\N}$. Hence $x \in J_{s,(j)}$ for some $j \in \N$. But then $f(x) > 0$ or $f(x)< 0$, so $x \notin E(f)$.
This ends the proof of equality $E(f) = F$.
\end{proof}

\section{An open problem}
It would be interesting to decide which situation is more common (in the Baire category sense)
for functions $f\in C[0,1]$:
when $\lambda(E(f))=0$ or when $\lambda(E(f))>0$.
Of course, we should restrict this question to functions $f\in C[0,1]$ for which $E(f)\neq\emptyset$.
Firstly, we need to assume $f^{-1}[\{0\}]\neq\emptyset$. Secondly, since $E(f)$ is nowhere dense,
it is natural to assume that the interior $\on{Int}(f^{-1}[\{0\}])$ is empty. Also, we should admit a situation where
$E(f)$ is of positive measure. To this aim, fix $\alpha\in (0,1)$ and assume that
$\lambda(f^{-1}[\{ 0\}])\ge \alpha$. Thus it is reasonable to consider the following set
\begin{equation} \label{equa}
ZC_\alpha [0,1]:=\{ f\in C[0,1]\colon f^{-1}[\{0\}]\neq\emptyset \mbox{ and }\on{Int}(f^{-1}[\{0\}])=\emptyset \mbox{ and }\lambda(f^{-1}[\{ 0\}])\ge\alpha\}.
\end{equation}

\begin{proposition} \label{propa}
$ZC_\alpha [0,1]$ is a Polish subspace of $C[0,1]$.
\end{proposition}
\begin{proof}
By the Alexandrov theorem, it suffices to show that $ZC_\alpha [0,1]$ is of type $G_\delta$.
Firstly, note that the set 
$\{f\in C[0,1]\colon f^{-1}[\{0\}]\neq\emptyset\}$ is open since its complement
$$\{f\in C[0,1]\colon\exists x\in [0,1],\; f(x)=0\}$$
 is closed which follows from the compactness of $[0,1]$.
 
Next, given a base $(U_n)_{n\in\N}$ of open sets in $[0,1]$, observe that 
$$\{ f\in C[0,1]\colon \on{Int}(f^{-1}[\{0\}])=\emptyset\}
=\bigcap_{n\in\N}\bigcup_{x\in U_n}\{f\in C[0,1]\colon f(x)\neq 0\}$$ 
and this set is of type $G_\delta$.

Finally, note that the set $\{ f\in C[0,1]\colon \lambda(f^{-1}[\{ 0\}])\ge\alpha\}$ is closed since
its complement can be expressed in the form 
$$\{ f\in C[0,1]\colon (\exists\; U\; \mbox{- open, }\lambda(U)<\alpha)\; f^{-1}[\{ 0\}])\subset U\}.$$
Here the set $\{f\in C[0,1]\colon f^{-1}[\{ 0\}]\subset U\}$ is open since 
$\{ f\in C[0,1]\colon (\exists\; x\in [0,1]\setminus U)\; f(x)=0\}$ is closed.
Consequently, by definition (\ref{equa}), the set $ZC_\alpha[0,1]$ is of type $G_\delta$.
\end{proof}

\begin{problem}
Given $\alpha\in (0,1)$, establish the Baire category of the following complementary sets in the Polish space $ZC_\alpha[0,1]$
$$\{f\in ZC_\alpha[0,1]\colon \lambda(E(f))=0\},\;\;\{f\in ZC_\alpha[0,1]\colon \lambda(E(f))>0\}.$$
\end{problem}

\section{Appendix: the proof of Theorem \ref{BV}}
\begin{proof}
Observe that $U=\bigcap_{n\in\N}U_n$
where
$$U_n:=\bigcup_{0=x_0<x_1<\dots <x_k=1}\left\{ f\in C[0,1]\colon \sum_{i=1}^k |f(x_i)-f(x_{i-1})|>n\right\} .$$
The sets $U_n$ are open, so $U$ is of type $G_\delta$. 

We will prove that each of the sets $C[0,1] \setminus U_n$ is porous.
Fix $f\in C[0,1]$ and $\ve >0$. For any $n \in \N$, we can find a polygonal function $g\in C[0,1]$ such that $||g||=\frac{\ve}{2}$ and
there is a partition $0=x_0<x_1<\dots <x_k=1$ of $[0,1]$ satisfying
\begin{equation} \label{dwa}
\Var(g,[0,1])=\sum_{i=1}^k|g(x_i)-g(x_{i-1})|>\Var(f,[0,1])+n.
\end{equation}
Indeed, it suffices to put $g(x):=\ve\on{dist}(\frac{k}{2}x, \Z)$ for $x\in [0,1]$ with $k\in 2\N$ such that $\frac{k\ve}{4} > \Var(f,[0,1])+n$. In the role of $x_i$'s, we take the first coordinates of vertices of the polygonal function $g$, that is, points of the form $\frac{i}{k}$ for $i =0,\dots,k$. Note that
$$\sum_{i=1}^k|g(x_i)-g(x_{i-1})| = \frac{k\ve}{2}.$$

Now, we will show that 
\begin{equation} \label{tre}
B\left(f+g, \frac{\ve}{8}\right) \subset B(f,\ve)\cap U_n. 
\end{equation}
We have $B(f+g,\frac{\ve}{8})=f+B(g,\frac{\ve}{8})$.
Fix any $h \in B(g, \frac{\ve}{8})$. Then 
$$||f+h-f||\le ||h-g||+||g||<\frac{\ve}{8}+\frac{\ve}{2}=\frac{5\ve}{8}<\ve.$$
This yields $B(f+g, \frac{\ve}{8}) \subset B(f,\ve)$.
For the taken $h$, put $\delta_i := h(x_i)-g(x_i)$ if $i=0,1,\dots,k$. Clearly, $\delta_i \in (-\frac{\ve}{8},\frac{\ve}{8})$. 
Thus 
$$\sum_{i=1}^k|h(x_i)-h(x_{i-1})|=\sum_{i=1}^k|g(x_i)+\delta_i-g(x_{i-1})-\delta_{i-1}| $$$$\geq \sum_{i=1}^k|g(x_i)-g(x_{i-1})|-\sum_{i=1}^k|\delta_{i-1}-\delta_i| > \frac{k\ve}{2} -k\frac{2\ve}{8}  = \frac{k\ve}{4} >\Var(f,[0,1])+n.$$
Consequently,
$$\sum_{i=1}^k|(f+h)(x_i)-(f+h)(x_{i-1})|  
\ge \sum_{i=1}^k|h(x_i)-h(x_{i-1})|-
\sum_{i=1}^k|f(x_i)-f(x_{i-1})|$$
$$>\Var(f,[0,1])+n-\Var(f,[0,1])=n.$$
Hence $B(f+g, \frac{\ve}{8}) \subset U_n$.

Now, by (\ref{tre}), it follows that $\gamma(f,\ve,C[0,1]\setminus U_n) \geq \frac{\ve}{8}.$ Thus, for any $f \in C[0,1]$ (in particular, for all $f \in C[0,1]\setminus U_n$) we have
$$p(C[0,1]\setminus U_n,f) =2 \limsup\limits_{\ve \to 0^+}  \frac{\gamma(f,\ve,C[0,1]\setminus U_n)}{\ve} \geq 2 \frac{\ve}{8\ve} = \frac{1}{4} > 0.$$
Therefore, the set $C[0,1]\setminus U_n$ is porous for all $n \in \N$. So, the set $C[0,1] \setminus U = \bigcup_{n \in \N} (C[0,1] \setminus U_n)$ is $\sigma$-porous as desired.
\end{proof}

From the above proof it follows that the set $C[0,1]\setminus U$ has even a stronger property. Namely,
each set $C[0,1]\setminus U_n$, $n\in\N$, has positive lower porosity at every point where
the lower porosity is defined as in formula (\ref{por}) with $\limsup$ replaced by $\liminf$.
This kind of porosity was considered in \cite{Za1}.

\end{document}